\documentclass{amsart}
\usepackage{amssymb}
\usepackage{amsfonts}
\usepackage{amsmath}
\usepackage{graphicx}

\usepackage{color}

\newtheorem{theorem}{Theorem}
\theoremstyle{plain}
\newtheorem{acknowledgement}{Acknowledgement}

\newtheorem{remark}{Remark}

\numberwithin{equation}{section}

\newcommand{\diag}{{\rm diag}}
\newcommand{\PP}{{\sf P\,}}
\def\ce{\operatorname{ce}}
\def\se{\operatorname{se}}
\def\trace{\operatorname{trace}}

\def\ff{{\overline f}}
\def\SEN{{\widehat{SE(2,N)}}}

\newcommand{\I}{ \bigl |}

\begin{document}
\title[hypoelliptic diffusion and human vision]{Hypoelliptic diffusion and human vision: \\
a semi-discrete new twist}
\author[U. Boscain]{Ugo Boscain}
\address{CNRS, CMAP, \'{E}cole Polytechnique CNRS, Route de Saclay, 91128 Palaiseau Cedex, France; INRIA Team GECO}
\email{boscain@cmap.polytechnique.fr}
\urladdr{http://www.cmapx.polytechnique.fr/\symbol{126}boscain/}
\author[R. Chertovskih]{Roman Chertovskih}
\address{Centre for Wind Energy and Atmospheric Flows, Faculdade de Engenharia da Universidade do Porto,
Rua Dr. Roberto Frias s/n, 4200-465 Porto, Portugal;
International Institute of Earthquake Prediction Theory and
Mathematical  Geophysics, Profsoyuznaya str. 84/32, 117997  Moscow, Russia}
\email{roman@mitp.ru}
\author[J.-P. Gauthier]{Jean-Paul Gauthier}
\address{LSIS, UMR\ CNRS 7296, Universt\'{e} de Toulon USTV, 83957, La Garde Cedex, France; INRIA Team GECO}
\email{gauthier@univ-tln.fr}
\urladdr{http://www.lsis.org/jpg}
\author[A. Remizov]{Alexey Remizov}
\address{CMAP, \'{E}cole Polytechnique CNRS, Route de Saclay, 91128 Palaiseau Cedex, France}
\email{alexey-remizov@yandex.ru}
\date{January 30, 2013}

\newcommand{\blue}{\color{blue}}
\newcommand{\red}{\color{red}}

\begin{abstract}
This paper presents a semi-discrete alternative to  the theory of neurogeometry of vision, due to Citti, Petitot and Sarti.  We propose a new ingredient, namely  working on the group of translations and discrete rotations $SE(2,N)$.

The theoretical side of our study relates the stochastic nature of the problem with the Moore group structure of $SE(2,N)$. Harmonic analysis over this group leads to very simple finite dimensional reductions.

We then apply these ideas to the inpainting problem which is reduced to the integration of a completely parallelizable finite set of Mathieu-type diffusions (indexed by the dual of $SE(2,N)$ in place of the points of the Fourier plane, which is a drastic reduction).

The integration of the the Mathieu equations can be performed by standard numerical methods for elliptic diffusions and leads to a very simple and efficient class of inpainting algorithms.

We illustrate the performances of the method on a series of deeply corrupted images.
\end{abstract}
\maketitle


\section{Introduction}

\subsection{History and pre-requisites}

In his beautiful book \cite{PET1}, Jean Petitot describes a sub-Riemannian model of the Visual cortex V1.

The main idea goes back  to the  paper \cite{hubel} by
H\"{u}bel an Wiesel in 1959 (Nobel prize in 1981)  who showed that in the
visual cortex V1, there are groups of neurons that are sensitive to
position and directions\footnote{ Geometers call ``directions'' (angles modulo $\pi$) what neurophysiologists call
``orientations''. }
with connections between them that are activated by the image.
The key fact is that the system of connections between neurons,
which is called the functional architecture of  V1,
preferentially connects  neurons detecting alignements.  This is the so-called  {\it pinwheels} structure of V1.

In 1989, Hoffman \cite{hoffman}  proposed to regard the visual cortex V1 as a manifold with a contact structure.  In  1998, Petitot \cite{petitot-tondut,PET} proposed to regard the  visual cortex V1 as a left-invariant sub-Riemannian structure on the Heisenberg group and gave an enormous impulse to the research on the subject. In 2006, Citti and Sarti \cite{CS} required the invariance under rototranslations and wrote the model on $SE(2)$.  In \cite{BDGR}, it was proposed to write the problem on $PT\mathbb{R}^2$ to avoid some topological problems and to be more consistent with the fact that the visual cortex V1 is sensitive only to directions (i.e., angles modulo $\pi$)  and not to directions with orientations (i.e., angles modulo $2\pi$).
Further development and elaboration of these ideas was carried on by various authors, 
see e.g., the papers~\cite{BSHIZ, WJ} and the references therein.

The detailed study of the sub-Riemannian geodesics was performed by Yuri Sachkov in a series of papers \cite{SAC,SAC1}.
This model was also deeply studied by Duits et al. in  \cite{DvA,DF,DF1} with  medical imaging applications in mind, and by  Hladky and  Pauls  \cite{HP}. Of course there are some close relations with the celebrated model by Mumford \cite{MUM}.

The main idea of the model (that in the following we call Citti-Petitot-Sarti's model) is  that V1 lifts the images $f(x,y)$ (i.e., functions of two position variables $x,y$ in the plane
$\mathbb{R}^{2}$ of the image) to functions over the projective
tangent bundle $PT\mathbb{R}^{2}$.
This bundle has as base $\mathbb{R}^{2}$ and as fiber over the point $(x,y)$ the set of directions  of straight lines lying on the plane and passing through $(x,y)$ and denoted by $\theta$.

In this model, a corrupted image is reconstructed  by minimizing the energy necessary to activate the regions of the visual cortex that are not excited by the image. If we think to an images as a set of (grey) level curves, each level curve is reconstructed by solving the following optimal control problem (in sub-Riemannian form)
\begin{eqnarray}
&&\Biggl(
\begin{array}
[c]{c}
\dot{x}(t)\\
\dot{y}(t)\\
\dot{\theta}(t)
\end{array}
\Biggr)=u(t)F(x,y,\theta)+v(t)G(x,y,\theta),
\label{e-1}\\
&&F(x,y,\theta)=\Biggl(
\begin{array}
[c]{c}
\cos(\theta)\\
\sin(\theta)\\
0
\end{array}
\Biggr),  \quad \
G(x,y,\theta)=
\Biggl(
\begin{array}
[c]{c}
0\\
0\\
1
\end{array}
\Biggr),\label{e-11}\\
&&\int\limits_0^T \biggl( u(t)^2+{\frac1\alpha} v(t)^2 \biggr) \,dt\to \min,\label{e-2}\\
&&\bigl(x(0),y(0),\theta(0)\bigr)=(x_0,y_0,\theta_0), \quad \ \bigl(x(T),y(T),\theta(T)\bigr)=(x_1,y_1,\theta_1).
\label{e-3}
\end{eqnarray}
Here $(x_0,y_0,\theta_0)$ and $(x_1,y_1,\theta_1)$ denote the initial and final condition in $PT\mathbb{R}^{2}$ of the level set that has to be reconstructed.
Equations \eqref{e-1}, \eqref{e-11} requires that the reconstructed curve is a lift of a planar curve.

In \eqref{e-2}, the final time $T$ should be fixed. The term $u(t)^2$ (respectively $v(t)^2$)  is proportional to the infinitesimal energy necessary to activate  
{\it horizontal} connections (resp.  {\it vertical}) connections.\footnote{We recall that neurons of V1 are grouped into orientation columns, each of them being sensitive to visual stimuli at a given
point  of the retina and for a given direction on it.
Orientation columns are themselves grouped into hypercolumns, each of them being sensitive
to stimuli at a given point  with any direction. Orientation columns are linked among them by two type of connections: vertical connections link orientation columns in the same hypercolumn, and
the  {\it horizontal} connections link orientation columns
belonging to different hypercolumns
and sensitive to the same orientation.
For an orientation column it is easy to activate another orientation column which is a ``first neighbor''  either by horizontal or by vertical connections.
}
The parameter $\alpha>0$ is a  relative weight. For more details on the sub-Riemannian problem see the original papers \cite{CS,petitot-tondut,PET}, the book \cite{PET1} or 
\cite{BCR,BDGR,cuspless, DBRS}  for a language more consistent with the one of this paper.

\begin{remark}
\label{remALFA}
From the theoretical point of view, the weight parameter $\alpha$ is irrelevant:
for any $\alpha >0$ there exists a homothety of the $(x,y)$-plane that maps geodesics
of the metric with the weight parameter $\alpha$ to those  of the metric with $\alpha=1$.
For this reason, in all theoretical considerations we fix
$\alpha=1$.
However its role will be important in our inpainting algorithms (see Section~\ref{weight}).
\end{remark}

The optimal control problem
\eqref{e-1}\,--\,\eqref{e-3}
described above  was used to reconstruct  smooth images (by reconstructing level sets) by Ardentov, Mashtakov and Sachkov \cite{ardentov}. The technique developed by them consists of reconstructing as minimizing geodesics the level sets of the image where they are interrupted.
Obviously, when applying this method to reconstruct images with large corrupted parts, one is faced to the problem that it is not clear how to put in correspondence the non-corrupted parts of the same level set.

To reconstruct a corrupted image (the V1 inpainting problem), it is natural to proceed as follows.
In the system \eqref{e-1}, excite all possible admissible paths in a stochastic
way. One gets a stochastic differential equation
\begin{equation}
\Biggl(
\begin{array}
[c]{c}
dx_{t}\\
dy_{t}\\
d\theta_{t}
\end{array}
\Biggr)
=
\Biggl(
\begin{array}
[c]{c}
\cos(\theta_{t})\\
\sin(\theta_{t})\\
0
\end{array}
\Biggr)
du_{t}
+
\Biggl(
\begin{array}
[c]{c}
0\\
0\\
1
\end{array}
\Biggr)
dv_{t},
\label{sys-stoch}
\end{equation}
where $u_{t}, v_{t}$ are two independent Wiener processes.

Consider the associated diffusion process (here we have fixed $\alpha=1$):
\begin{equation}
\frac{\partial \psi}{\partial t}= \frac{1}{2}\Delta \psi, \quad
\Delta = F^2 + G^2 = \biggl(\cos(\theta)\frac{\partial}{\partial x}+
\sin(\theta)\frac{\partial}{\partial y}\biggr)^{2} +
\frac{\partial^{2}}{\partial\theta^{2}}. \label{contdif}
\end{equation}
The operator $\Delta$ is not elliptic, but it is hypoelliptic (satisfies H\"{o}rmander
condition). In this setting, equation \eqref{contdif} was proposed first by Citti and Sarti in \cite{CS}.

To be able to use equation \eqref{contdif} to reconstruct a corrupted image, one has to specify two things: {\bf i)} how  to lift the corrupted 
image on $PT\mathbb{R}^2$ to get an initial condition for \eqref{contdif};  {\bf ii)}  how to project the result of the diffusion on $\mathbb{R}^2$ to get the reconstructed image.

Different  procedures for these steps have been proposed in \cite{BDGR,CS,DF,DF1,SC}. A clever lifting process was proposed  in \cite{DF,DF1}.

\begin{remark} {\bf Relation between the sub-Riemannian problem and the diffusion equation \eqref{contdif}}: 
\begin{itemize}
\item By the Feynman\,--\,Kac formula, integrating Equation~\eqref{contdif},  one expects to reconstruct the most probable missing level curves (among admissible).
\item The solution of \eqref{contdif} is strictly related to the solution of  the optimal control problem  \eqref{e-1}\,--\,\eqref{e-3}.
Indeed, a result by Leandre \cite{leandre1,leandre2} shows that $-t\, \mbox{log}(P_t(q,\bar q))$ tends (up to constants) to $E(q,\bar q)$,  
as $t \to 0$, where $P_t(q,\bar q)$ is the kernel of Equation \eqref{contdif} and $E(q,\bar q)$ is the value function for the problem \eqref{e-1}\,--\,\eqref{e-3}.
Here $q,\bar q\in PT\mathbb{R}^2$.
\end{itemize}
\end{remark}

There are at least two parameters to be suitably tuned: the time of the diffusion and the parameter $\alpha$.

The problem of the numerical integration of \eqref{contdif} is subtle since multiscale sub-Riemannian effects are hidden inside\footnote{For instance, one cannot use 
the nice method presented in \cite{achdou} for the Heisenberg group, since in $SE(2)$ one cannot build a refinable grid having a subgroup properties. 
Indeed, a grid that is subgroup of $SE(2)$ has at most 6 angles and hence cannot be refined, see e.g.,~\cite{brezzloff}.}
and has been approached in different way. In \cite{CS,SC}
the authors use a finite difference discretisation of all derivatives. In \cite{DvA,DF,DF1} the authors use an almost explicit expression for the heat kernel. 
In \cite{BDGR}  some promising results about inpainting using hypoelliptic diffusion were obtained via a sophisticated algorithm based upon the generalized Fourier 
transform on the group $SE(2)$. See also \cite{az,zw}.

The purpose of this paper is to approach the problem differently: we will not present a new method to solve numerically \eqref{contdif}, 
but rather a variant to the Citti-Petitot-Sarti model by assuming that the visual cortex can detect only a finite number of directions. 
This idea will permit to put the problem in a very different theoretical context and to overcome most of the difficulties in the numerical integration of \eqref{contdif}.

\subsection{Contents of the paper}

\subsubsection{A semidiscrete model}
In this paper, we introduce a new ingredient, namely we conjecture the following:

\begin{itemize}
\item the visual cortex can detect  a finite (small) number $N$ of directions only.
\end{itemize}

This conjecture is based on the observation of the organization of the visual cortex in pinwheels. The planar and the angular degrees of freedom cannot be treated in the same way. We conjecture that there are topological constraints that prevent the possibility of detecting a continuum of directions even when sending the distance between pinwheels to zero.

In this paper we are not  going to justify this conjecture, but   we can mention several previous contributions: from neurophysiology (see \cite{neuro1} and references therein) and from image processing engineering  (see \cite{lee}  and references therein).
Another justification from the present paper is just the observation that increasing $N$ does not lead to significant improvement in our reconstruction process (see Section~\ref{heursec}).

Rather, we concentrate  on the effect of this conjecture on the mathematical model.

This assumption being made, the group $SE(2,N)$ appears naturally and rather elementary stochastic analysis leads to the equivalent \eqref{FKP} below of the hypoelliptic diffusion \eqref{contdif}. These topics
are the purpose of Section \ref{SE2N} and  \ref{s-semi-d}, while
in Section \ref{s-se2} we recall shortly what happens considering $SE(2)$.

Here some points have to be stressed that are in order:
\begin{enumerate}
\item As usual for the treatment of invariant diffusions on topological groups, the basic tool is an analog of the usual Fourier Transform, i.e., the GFT, see the appendix \ref{GFT}.
\item The nature of $SE(2,N)$ is completely different from that of $SE(2)$: although non-abelian and non-compact, it is a Moore group, which is a special case of maximally almost periodic (MAP) groups (having the Chu duality property). See  appendix \ref{map} for a short introduction to these concepts.
\item The expression \eqref{Discrker} of the heat kernel on $SE(2,N)$ is much simpler than the one \eqref{kernel1} over $SE(2)$.
\item Although we use a ``small number of directions'' we show the following coherence property: when $N$ tends to infinity, the diffusion \eqref{FKP} tends to the diffusion \eqref{contdif} over $SE(2)$.

\end{enumerate}

\subsubsection{A first class of algorithms that is not suitable for our purposes}
In Section~\ref{prea} we discuss the use of the MAP property of $SE(2,N)$ for the integration of our semi-discrete diffusion \eqref{FKP}. The group $SE(2,N)$ being MAP, it follows (exactly as for the standard heat equation over $\mathbb{R}^n$) that the space of almost periodic functions is a good candidate for the space of solutions. In fact, restricting to finite dimensional subspaces (called $SE(2,N,K)$ below) of finite trigonometric polynomials (conceptually finite linear combinations of coefficients of unitary irreducible representations of $SE(2,N)$), an immediate computation leads to Equation~\eqref{redap}. This equation  shows (Theorem~\ref{th1}) that the integration of the diffusion on $SE(2,N)$ reduces to solving a finite (but large) set of finite dimensional linear evolution equations.

These equations can be solved in parallel, which fits with the structure of the visual cortex V1, which is thought to make a lot of parallel computations. See, for instance, \cite{parallel}.

It turns out that this space of solutions is not very relevant for the purpose of image reconstruction, mostly for two reasons:

\begin{enumerate}
\item it is not adapted to the usual simple representation of an image as a table of numbers (grey levels at points of a grid);
\item due to the lifting step the functions under considerations are far from being almost-periodic. In particular, in the limit $SE(2)$ case they are distributions as explained above. In the $SE(2,N)$ case, they are not continuous.
\end{enumerate}

\subsubsection{A final inpainting algorithm based on the semi-discrete model}

The two items just above are the reasons for the final algorithm proposed in Section \ref{algo}. Since a  grey-level  image is usually given under the guise of a (square, say)  table of positive real numbers, we will consider a spatial discretisation of the continuous part of the operator, coherent with this given grid. This discretisation produces the equation \eqref{deceq} which is exactly of the same type of equation \eqref{redap}.
This leads to Theorem~\ref{th2}: the integration of the spatially discretised operator over $SE(2,N)$ reduces to a large finite set of linear finite dimensional evolution equations of Mathieu type that can be  solved in parallel.

The following observation is important: the number of these Mathieu type diffusions reflects the structure of (the dual of) $SE(2,N)$ and not the number of points in the image.

\subsubsection{Heuristic complements and numerical results}

In Section \ref{heursec}, we show a number of numerical results.

In fact, in our algorithm, at this point, several steps are not fixed. We have chosen for these steps the most simple and maybe naive solutions, namely:
\begin{itemize}
\item The lifting procedure to $SE(2,N)$ that we implement is the most simple that can be thought: after smoothing the image (Gaussian smoothing), 
    the gradient at the point is well defined (both the smoothing and the gradient computation are performed at the level of the FFT of the image), 
    and we just lift to the closest value of the angle.\footnote{However, it has been pointed out by a referee that certain more effective lifting procedures 
    (i.e., adapted to the $SE(2)$ or even $SE(2,N)$ structures), can be developed, see~\cite{DF,DF1}. But this is beyond the scope of this paper.
    }

To simplify, even when we know where the image is corrupted we always compute the gradient in this rough way. This point could certainly be improved on. The only place where we eventually use the knowledge of the corrupted part is inside the integration procedure. See the concepts of static and dynamic restorations in Section~\ref{heursec}.

\item The projection step. There are at least two obvious possibilities: projecting by taking
either the maximum, or the average, over angles. Numerous
experiments show that the projection made by the maximum provides
better results. For less rough ideas, see \cite{DFG,forssen}.

\item The numerical resolution of the Mathieu type diffusions: they are finite dimensional evolution equations with elliptic second term. We have chosen the Crank-Nicolson method which is recommended in such a situation. We do not intend to discuss the convergence of this method to the solution to the initial hypoelliptic diffusion \eqref{contdif},
this is of no interest within the scope of this paper.
\end{itemize}

In this Section \ref{heursec} we present also some heuristic improvements of the
algorithm in the case where we can distinguish between the corrupted
and non-corrupted parts of the image.

\subsubsection{End of the paper}
In Appendix~\ref{APP6}, we present a series
of results of reconstruction with high corruption rates  and the table of all parameters
used for the reconstruction procedures. These parameters include: the time of diffusion, the weight $\alpha$, and other parameters that are used to ``tune'' the action of the diffusion on the corrupted and non corrupted points.

Finally,  Appendix~\ref{APP7}  is devoted to basic theoretical concepts and to certain technical computations.

\medskip

We do not pretend that our results are better than the inpainting algorithms existing today. We just claim that they really seem to validate our semi-discrete version of the Citti-Petito-Sarti model. Moreover, we emphasize the global character of the basic algorithm.

To finish let us underline the main point. The most interesting feature of the ``algorithms'' resulting from the present  study is the following:  
they are naturally massively parallel, which reflects the structure of the dual of the group $SE(2,N)$. 



\section{The Operators and the groups under consideration\label{sec1}}

\subsection{Groups}

We advise the uninitiated reader to start with our paper \cite{ABG}
and to have a look to Appendix~\ref{APP7} at the end of this paper.

\subsubsection{The group of Motions SE(2)\label{s-se2}}

The group law over the Lie group $SE(2)$ has multiplication law
$(X_{2},\theta_{2}) \cdot (X_{1},\theta_{1}) =
(X_{2}+R_{\theta_{2}}X_{1}, \theta_{1}+\theta_{2})$, where
$$
X=\left(
\begin{array}
[c]{c}
x\\
y\\
\end{array}
\right),
\quad
R_{\theta}=\left(
\begin{array}
[c]{cc}
\cos(\theta) & \sin(\theta) \!\\
\!\!-\sin(\theta) & \cos(\theta) \!\\
\end{array}
\right),
$$
$R_{\theta}$ is the rotation of angle $\theta$ in the $X$-plane.

The (strongly continuous) unitary irreducible representations of
$SE(2)$ are well known. However for analogy with the next section,
we need to recall basic facts. For a survey, see \cite{VIL, WAR}.

As representations of a semi-direct product, they can be obtained by
using Mackey's imprimitivity theorem, and therefore, they are
parametrized by the orbits of the (contragredient) action of
rotations on the $X$-plane, i.e., they are parametrized by the half
lines passing through the origin. Additionally, corresponding to the
origin, there are the characters of the rotation group $S^{1}$ that
do not count in the support of the Plancherel's measure. Finally,
the representation $\chi_{\lambda}$ corresponding to the circle of
radius $\lambda$ acts over $L^{2}(S^{1}),$ and is given by
\begin{equation*}
[\chi_{\lambda}(X,\theta) \cdot \varphi](u)=e^{i \langle
v_{\lambda},R_{u}X \rangle }\varphi(u+\theta),
\end{equation*}
where $i = \sqrt{-1}$, $v_{\lambda}$ is a vector of length $\lambda$ in $\mathbb{R}^{2}$, and
$\langle \cdot,\cdot \rangle$ denotes the standard Euclidean scalar product over $\mathbb{R}^{2}$.

Note that $SE(2)$ is far from being MAP,
since all its finite dimensional unitary irreducible representations
are given by the characters of $S^{1}$ only.

Let us also recall \cite{ABG, BDGR, DvA, DF}
that  the sub-Riemannian heat kernel on $SE(2)$, corresponding to the hypoelliptic Laplacian $\frac{1}{2}
\Delta$ defined in \eqref{contdif} is given by\footnote{in this formula $\alpha=1$}
\begin{multline}
P_{t}(X,\theta)  = \frac{1}{2} \int\limits_{0}^{+\infty}
\biggl(
\sum\limits_{n=0}^{+\infty}
e^{a_{n}^{\lambda}t} \biggl\langle \ce_{n}\biggl(\theta,\frac{\lambda^{2}}{4}\biggr), \, \chi_{\lambda}(X,\theta) \ce_{n}\biggl(\theta,\frac{\lambda^{2}}{4}\biggr) \biggr\rangle + \\
\sum\limits_{n=0}^{+\infty}
e^{b_{n}^{\lambda}t} \biggl\langle \se_{n}\biggl(\theta,\frac{\lambda^{2}}{4}\biggr), \, \chi_{\lambda}(X,\theta) \se_{n}\biggl(\theta,\frac{\lambda^{2}}{4}\biggr) \biggr\rangle \,
\lambda d\lambda,
\label{kernel1}
\end{multline}
where the functions $\se_{n}$ and $\ce_{n}$ are the $2\pi$-periodic
Mathieu sines and cosines, and
$a_{n}^{\lambda}=-\frac{\lambda^{2}}{4}-a_{n}\bigl(\frac{\lambda^{2}}{4}\bigr)$,
$b_{n}^{\lambda}=-\frac{\lambda^{2}}{4}-b_{n}\bigl(\frac{\lambda^{2}}{4}\bigr)$,
with $a_{n},b_{n},$  the characteristic values for the Mathieu
equation.

Then the solution of the diffusion equation \eqref{contdif} with the
initial condition $\psi\I_{t=0} = \psi_0$ is given by the
right-convolution formula:
\begin{equation}
\psi(t,X,\theta) = e^{t\Delta} \psi_{0}(X,\theta) =
\left(\psi_{0} \ast P_{t} \right) (X,\theta) \, = \! \int\limits_{SE(2)}
\psi_{0}(g)P_{t}\bigl(g^{-1} \! \cdot (X,\theta)\bigr)\,dg,
\label{difsolution}
\end{equation}
where $g \in SE(2)$ and $dg$ is the Haar measure: $dg = dxdyd\theta$.

\begin{remark}
In \cite{DvA,DF}, the authors reduce this expression to a simpler one involving only 4 non-periodic Mathieu functions.  
The periodic Mathieu expansion \eqref{kernel1}  is more directly related with the Fourier transformations over $SE(2)$ and $SE(2,N)$, 
the periodic Mathieu functions being eigenfunctions of the relevant Mathieu operators arising by applying the GFT  
to the original diffusion equation \eqref{contdif}. 
\end{remark}

\subsubsection{The group of discrete motions $SE(2,N)$\label{SE2N}}

In the next section, a discrete-angle version of the above diffusion
equation will appear naturally. It corresponds to the group of
motions $SE(2)$ restricted to rotations with angles
$\frac{2k\pi}{N}$. This group is denoted by $SE(2,N)$. It has very
special features: it is MAP,
and all its unitary irreducible representations are finite dimensional (it is a
Moore group, see \cite{HEY} for details), although it is not
compact. In particular, this follows from \cite[16.5.3, page~304]{DIX}: it is the semi-direct product of the compact subgroup
$K=\mathbb{Z}/N\mathbb{Z}$, and a normal subgroup $V$ isomorphic to
$\mathbb{R}^{2}$, each element of $V$ commuting with the connected
component of the identity (which in this case is $V$ itself). To
simplify, we set $R_{k}=R_{\frac{2k\pi}{N}}$.

Regarding the representations of $SE(2,N)$, using Mackey's machinery the picture is the following.
Besides the characters of $K$, the representations coming in the
support of the Plancherel measure are parametrized by $(\lambda,\nu)
\in \mathbb{R}_{\ast}^{+} \times [0, \frac{2\pi}{N}[ \, =\! \SEN,$
the \textbf{dual}\footnote{ The natural ``dual topology'' of $\SEN$
is that of a cone, that consists of considering
$\mathbb{R}_{\ast}^{+}\times [0,\frac{2\pi}{N}]$ and identifying
($\lambda,0)$ with $(\lambda,\frac{2\pi}{N})$. } of $SE(2,N)$. They
act on $\mathbb{C}^{N}$ and are given by
\begin{equation}
\chi_{\lambda,\nu}(X,r) = \diag_{k} \bigl(e^{i \langle
V_{\lambda,\nu},R_{k}X \rangle }\bigr) S^{r},
\label{repN}
\end{equation}
where $\diag_{k} \bigl(e^{i \langle V_{\lambda,\nu},R_{k}X \rangle
}\bigr)$ is the diagonal  matrix with diagonal elements $e^{i
\langle V_{\lambda,\nu},R_{k}X \rangle}$,
$V_{\lambda,\nu}=(\lambda\cos (\nu),\lambda\sin(\nu))$, $k=1,
\ldots, N$, and $S$ is the shift matrix over $\mathbb{C}^{N}$, i.e., $Se_{k}=e_{k+1}$ for $k=1, \ldots, N-1$, and $Se_{N}=e_{1}$,
where $e_1, \ldots, e_N$ are the elements of the standard basis in $\mathbb{C}^{N}$.

\subsection{The semi-discrete diffusion operator\label{s-semi-d}}

\subsubsection{Semi-discrete versus continuous}

Firstly, we show that a certain semi-discrete (discretization with respect
to the angle) model of the diffusion is compatible with the limit
continuous model. For the considerations in this section, one may
refer to the paper \cite{AHD}. The developments in this section may be considered as elementary by probabilistic readers, however they are very
instructive with respect to the nature of our semi-discrete model. Therefore we give the details.

The diffusion equation \eqref{contdif} comes from the stochastic
differential equation \eqref{sys-stoch}, where $u_{t},v_{t}$ are two
independent standard Wiener processes. It is the associated Fokker-Planck (or Kolmogorov
forward) equation to \eqref{sys-stoch}.

From the image processing point of view, integrating the diffusion
is equivalent to excite all possible admissible paths, in a
stochastic way.

\medskip

%

In fact, in the real structure of the V1 cortex,
\textbf{a finite (small) number of angles only is taken into account}. Here this number is denoted by~$N$.

Therefore, it is natural to consider the following stochastic
process with jumps evidently connected with the stochastic equation
\eqref{sys-stoch}:
\begin{equation}
dz_{t}=
\left(
\begin{array}
[c]{c}
\! dx_{t} \!\\
\! dy_{t} \! \\
\end{array}
\right)
=
\left(
\begin{array}
[c]{c}
\! \cos(\theta_{t})\! \\
\! \sin(\theta_{t})\! \\
\end{array}
\right)
dv_{t},
\label{sto-proc}
\end{equation}
in which $\theta_{t}$ is a jump process. Set
$\Lambda_{N}=(\lambda_{i,j})$, $i,j=0,\ldots,N-1$, where
$$
\lambda_{i,j}=\lim_{t\to 0} \frac{ \PP [\theta_{t}=e_{j}|\theta_{0}=e_{i}]}{t} \ \ \textrm{for} \ \   i\neq j, \quad
\lambda_{j,j} =-\sum_{i \neq j} \lambda_{i,j}.
$$
The matrix $\Lambda_{N}$ is the infinitesimal generator of the jump
process.

We assume Markov processes, where the law of the first jump time is the exponential distribution with the rate parameter $\beta>0$
(the value of $\beta$ will be specified later on).
The jump has probability $\frac{1}{2}$ on both sides. Then we get a Poisson process, and the probability of $k$ jumps between $0$ and $t$ is
$$
\PP [k \ \text{jumps}] = \frac{(\beta t)^{k}}{k!}e^{-\beta t}.
$$
So that:
\begin{align*}
\PP [\theta_{t}  &  =e_{i\pm 1}|\theta_{0}=e_{i}]=\frac{1}{2} \bigl( \beta t+kt^{2} + o(t^{2}) \bigr)e^{-\beta t},\\
\PP [\theta_{t}  &  =e_{i\pm 2}|\theta_{0}=e_{i}]=\frac{1}{4} \biggl(\frac{1}{2}\beta^{2}t^{2} + o(t^{2}) \biggr)e^{-\beta t},\\
\PP [\theta_{t}  &  =e_{i\pm n}|\theta_{0}=e_{i}]= O(t^{n})
e^{-\beta t}, \ \ n =2,3, \ldots,
\end{align*}
with the convention that $e_{i}$ is modulo $N$.

Hence $\lambda_{i,i\pm 1}=\frac{1}{2}\beta$ and
$\lambda_{i,i}=-\beta$. All other elements of the matrix
$\Lambda_{N}$ are equal to zero. Then, the infinitesimal generator
of the semi-group associated with the stochastic process
$(z_{t},\theta_{t})$ is of the form:
\begin{equation}
\mathcal{L}_{N}\Psi(z,e_{i})=(A\Psi)_{i}(z)+(\Lambda_{N}\Psi(z,e_{i}))_{i},
\label{LNPsi}
\end{equation}
where $\Psi_{j}(z)=\Psi(z,e_{j})$, $z=(x,y)$, (here we use $\Psi$ instead of $\psi_t$ to distinguish the discrete case and to have a more compact notation) and
\begin{equation}
(A\Psi)_{i}(z)  =
A\Psi(z,e_{i})=\frac{1}{2}\biggl(\cos(e_{i})\frac{\partial
}{\partial x}+\sin(e_{i})\frac{\partial}{\partial y}\biggr)^{2}
\Psi(z,e_{i}), \label{laN}
\end{equation}
\begin{equation}
(\Lambda_{N}\Psi(z,e_{i}))_{i}
=\sum_{j=0}^{n-1}\lambda_{i,j}\Psi_{j}(z)=\frac{\beta}{2}\bigl(\Psi_{i-1}(z)-2\Psi_{i}(z)+\Psi_{i+1}(z)\bigr),
\label{laN1}
\end{equation}
where the subscript $i$ means the $i$-th coordinate of the vector.

Therefore, if we set $\beta=\bigl(\frac{N}{2\pi}\bigr)^{2}$, we get
$$
(\Lambda_{N}\Psi(z,e_{i}))_{i} =
\frac{1}{2}\frac{\Psi_{i-1}(z)-2\Psi_{i}(z)+\Psi_{i+1}(z)}{\bigl(\frac{2\pi}{N}\bigr)^{2}}
= \frac{1}{2}\frac{\partial^{2}}{\partial\theta^{2}}\Psi(z,e_{i}) +
O\Bigl(\frac{1}{N}\Bigr)^2.
$$
At the limit $N \to \infty$, from formulae \eqref{LNPsi} --
\eqref{laN1} we get the second order differential operator
$$
\mathcal{L}\Psi(z,\theta)=\frac{1}{2}\biggl(\biggl(\cos(\theta)\frac{\partial}{\partial
x} +\sin(\theta)\frac{\partial}{\partial y}\biggr)^{2} +
\frac{\partial^{2}}{\partial\theta^{2}}\biggr) \Psi(z,\theta) =
\frac{1}{2} \Delta \Psi(z,\theta),
$$
that is, the operator of our diffusion process \eqref{contdif}.
However the exact Fokker-Planck equation\footnote{Note that here,
due to self adjointness, the Kolmogorov-backward equation is the
same as the Kolmogorov-forward one (Fokker-Planck).} with number of
angles $N \ll \infty$ contains the parameter~$\beta$:
\begin{multline}
\label{FKP}
\frac{d\Psi_r}{dt}(t,z) =\frac{1}{2} \biggl(\cos(e_r)\frac{\partial}{\partial x} +
\sin(e_r)\frac{\partial}{\partial y}\biggr)^{2}\Psi_r(t,z)+ \\
\frac{\beta}{2}\bigl(\Psi_{r-1}(t,z)-2\Psi_{r}(t,z)+\Psi_{r+1}(t,z)\bigr),
\quad r = 0, \ldots, N-1.
\end{multline}

\subsubsection{The semi-discrete heat kernel via the GFT\label{DK}}

The GFT (generalized non-commutative Fourier transform, see
\cite{ABG} and Appendix~\ref{APP7})
transforms our hypoelliptic diffusion
equation into a continuous sum of diffusions with elliptic
right-hand term, and the summation over $\SEN$ is with respect
to the Plancherel measure $\lambda d\lambda d\nu$. We
compute the semi-discrete heat kernel via the GFT,
just as in \cite{ABG}, and we get a similar but simpler formula than
for the ``continuous'' heat kernel~\eqref{kernel1} in the case of the
group $SE(2)$:

Formulae \eqref{gft} and \eqref{gftinv} for the
direct and the inverse GFT show that, if we set
\begin{equation}
\label{restf} \widetilde{A}_{\lambda,\nu} = \Lambda_{N} - \diag_{k}
\bigl(\lambda^{2}\cos^{2}(e_{k}-\nu)\bigr),
\end{equation}
we get the following expression for the ``semi-discrete'' (or ``jump'') heat kernel:
\begin{equation}
D_{t}(z,e_{r})= \! \int\limits_{\SEN} \! \trace
\Bigl(e^{\widetilde{A}_{\lambda,\nu}t} \cdot \diag_{k}\Bigl(e^{i
\langle V_{\lambda,\nu},R_{k}z \rangle }\Bigr) S^{r} \Bigr) \lambda
d\lambda d\nu. \label{Discrker}
\end{equation}

Note also that it is a closed formula similar to the one in the
Heisenberg case (explicit  usual functions modulo a Fourier
transform). These are the only cases we know for noncompact groups,
where the kernel is obtained with such an explicit formula.


\subsubsection{A direct way to compute the semi-discrete heat kernel\label{DKbis}}

We start from our semi-discrete heat equation~\eqref{FKP}. The heat
kernel is the convolution kernel associated with $D_{t}(z,e_{r})$,
the fundamental solution of Equation~\eqref{FKP}.

Applying the ordinary Fourier Transform with
respect to the space variable~$z$ to~\eqref{FKP}, we get the
ordinary linear differential equation
\begin{equation}
\frac{d}{dt} \widetilde{D}_t(z) = \widetilde{A}_{\lambda,\nu}
\widetilde{D}_t(z), \label{new1}
\end{equation}
where $\widetilde{D}_t(z) = \bigl( D_t(z,e_1), \ldots,
D_t(z,e_N)\bigr)$, with the initial condition obtained from the
Dirac delta function at the identity by the ordinary Fourier
Transform:
$$
\widetilde{D}_0(z) = \delta_N = (0, \ldots, 0, 1).
$$

Then, taking the inverse ordinary Fourier Transform with respect to the
space variable, we get a second expression for the heat kernel:
\begin{equation}
D_{t}(z,e_{r})= \! \int\limits_{\mathbb{R}^2} \!
\Bigl(e^{\widetilde{A}_{\lambda,\nu}t} \delta_N \Bigr)_r \, e^{i
\langle V_{\lambda,\nu},z \rangle} \lambda d\lambda d\nu.
\label{new2}
\end{equation}

\begin{remark}
 Looking at the formulae \eqref{Discrker} and \eqref{new2}, it is not clear that these expressions are identical.
The proof of this fact is given in Appendix~\ref{kerproof}.
Although less direct, we prefer formula~ \eqref{Discrker}, that reflects more the structure of $SE(2,N)$.
\end{remark}


\subsection{The weighting of the metric\label{weight}}

Consider the diffusion process
\begin{equation}
\frac{\partial \psi}{\partial t}= \frac{1}{2}\Delta_{\alpha} \psi,
\quad \Delta_{\alpha} = F^2 + \alpha G^2 =
\biggl(\cos(\theta)\frac{\partial}{\partial x}+\sin(\theta)\frac{\partial}{\partial y}\biggr)^{2} + \alpha \frac{\partial^{2}}{\partial\theta^{2}}
\label{contdif-weight}
\end{equation}
with the weighting coefficient $\alpha>0$. The hypoelliptic operator
$\Delta$ defined in \eqref{contdif} is the case $\alpha=1$. Although
the coefficient $\alpha$ is theoretically irrelevant (see Remark~\ref{remALFA}),
it plays an essential role in practice.

On the same way, for the semi-discrete operator, we have
\begin{multline*}
\Delta^{(N)} \psi_{i}(t,z) = \\
\biggl(\cos(e_{i})\frac{\partial}{\partial x}+\sin(e_{i})\frac{\partial}{\partial y}\biggr)^{2}\psi_{i}(t,z)+
\beta_N \bigl(\psi_{i-1}(t,z)-2\psi_{i}(t,z)+\psi_{i+1}(t,z)\bigr) = \\
\biggl( \biggl(\cos(e_{i})\frac{\partial}{\partial
x}+\sin(e_{i})\frac{\partial}{\partial y}\biggr)^{2} + \beta_N
\Bigl(\frac{2\pi}{N}\Bigr)^2 \frac{\partial^2}{\partial
\theta^2}\biggr) \psi_{i}(t,z) +  O\Bigl(\frac{1}{N}\Bigr)^2
\end{multline*}
as $N \to \infty$. Comparing the above formula with \eqref{FKP}, we
have the relation:
\begin{equation}
\alpha = \beta_N \Bigl(\frac{2\pi}{N}\Bigr)^2. \label{relation}
\end{equation}
It appeared clearly in all our experiments that $N=30$ is always
enough (it brings nothing visible in the experiments to take
$N>30$).

\textbf{Both parameters $N$ and $\beta_N$ have a physiological
meaning}: the number $N$ is the number of (significant) directions in our model and $\beta_N$ is a  parameter specifying the dynamics of the jumps to closest neighbour direction. Therefore, the coefficient $\alpha$ of the limit behavior
($N \to \infty$) can certainly be obtained from physiological
considerations. However we  could not find any result from neurophysiological literature in this direction.

\begin{remark} {\bf(about projectivization)}
The orientation maps of the V1 cortex show ``pinwheels'', that is,
singularities where iso-direction lines of any orientation
$\theta$ converge. Pinwheels have ``chirality'', that is, they
rotate as $2\theta$ or $-2\theta$ along a small direct loop around
them. This could be reflected in formula \eqref{kernproj} below. See Figure \ref{f-pinwheels}.

\begin{figure}[ptb]%
\centering
\includegraphics[width=200pt,height=150pt]{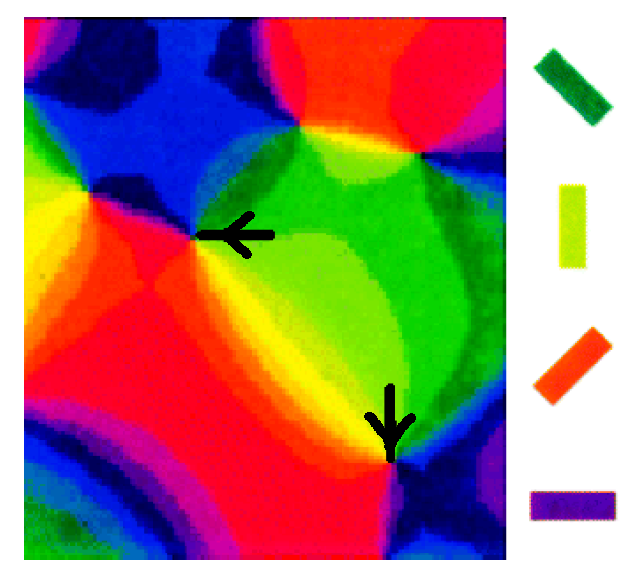}
\caption{Two pinwheels with opposite chirality.
 Each color corresponds to the sensitivity to one direction.}%
\label{f-pinwheels}%
\end{figure}

As we said, the neurons are not sensitive to the angles themselves,
but to directions only  (i.e., the angles are modulo $\pi$ and not
$2\pi$). It is the reason why we work in the projectivisation
$PT\mathbb{R}^{2}$ of the tangent bundle of $\mathbb{R}^{2}$,  in
place of $SE(2)$ itself. From the discrete point of view, it means
that if $N$ is the number of values of directions (not angles),
the angle-step is in fact $\frac{\pi}{N}$.

Also, as it was explained in \cite{BDGR}, the sub-Riemannian
structure over $PT\mathbb{R}^{2}$  itself is not trivializable. This
is not a problem for the operators $\Delta$ and $\Delta^{(N)}$,
since the functions $\cos(\theta)^{2}$, $\sin(\theta)^{2}$ and
$\sin(2\theta)$ are $\pi$-periodic. Hence, details relative
to this projectivisation are omitted in the following sections.

From the point of view of the heat kernels, in fact, if $p_{t}$
denotes the heat kernel over $PT\mathbb{R}^{2}$ (which is not a
group convolution kernel anymore, since $PT\mathbb{R}^{2}$ is not a
group), we have the formula
\begin{equation}
p_{t}((x,y,\theta),(\bar{x},\bar{y},\bar{\theta})) =
P_{t}((\bar{x},\bar{y},\bar{\theta})^{-1} \cdot (x,y,\theta)) +
P_{t}((\bar{x},\bar{y},\bar{\theta})^{-1} \cdot (x,y,\theta+\pi)),
\label{kernproj}
\end{equation}
where the inverses and products are intended in the group $SE(2)$.

A similar formula holds for the semi-discrete kernels $d_{t}$ and
$D_{t}$ from \eqref{Discrker} for even~$N$.
\end{remark}


\section{A view through almost periodic functions\label{prea}}

As we said in section \ref{SE2N}, the group $SE(2,N)$ is MAP.
It follows from the expression~\eqref{repN} of the
unitary irreducible representations that the Bohr-almost periodic
functions $f(x,y,r)$ are just those such that the functions
$f_{r}(x,y)=f(x,y,r)$, $r=1,\ldots,N$, are Bohr-almost periodic over
$\mathbb{R}^{2}$ in the usual sense. We call $AP(N)$ the set of
almost periodic functions on $SE(2,N)$, and we identify the elements
of $AP(N)$ to $\mathbb{C}^{N}$-valued functions whose components are
almost periodic over $\mathbb{R}^{2}$,  i.e., functions that are
uniform limits of trigonometric polynomials in the two variables
$(x,y)$. See Appendix \ref{map}.

These functions are dense among continuous functions over any
compact subset of $SE(2,N)$, and $AP(N)$ is a good candidate for the
space of solutions of our heat equation: exactly as for the usual
heat equation (see \cite[pp.~144--146]{CGB} for instance), the
uniformly bounded solutions of our heat equation with initial
conditions in $AP(N)$ remain almost periodic over $SE(2,N)$
uniformly in time.

Functions coming from standard images have support in a bounded subset of $SE(2,N)$.
They can be (after smoothing and lifting)
approximated uniformly over this bounded subset by trigonometric
$\mathbb{C}^{N}$-valued polynomials $Q(x,y)$ with components $Q_{r}(x,y)$:
\begin{equation}
Q_{r}(x,y)=
\sum\limits_{s=0}^{k} a_{r,\lambda_{s},\mu_{s}}e^{i(\lambda_{s}x+\mu_{s}y)},
\quad r = 0, \ldots, N-1,
\label{pol}
\end{equation}
where $i = \sqrt{-1}$.
The vector space of such $\mathbb{C}^{N}$-valued polynomials, for a
fixed finite number of distinct values of $\lambda_{s},\mu_{s}$,
i.e., $\omega_{s} = (\lambda_{s},\mu_{s})\in K$, a fixed finite subset
of $\mathbb{R}^{2}$, is denoted by $SE(2,N,K)$.

A straightforward computation shows that semi-discrete hypoelliptic equation
\eqref{FKP} restricts to $SE(2,N,K)$. It becomes
\begin{multline}
\frac{da_{r}^{s}}{dt}=-\frac{1}{2} \biggl(
\lambda_{s}\cos(\theta_{r})+\mu_{s}\sin(\theta_{r})\biggr)^{2}a_{r}^{s} +
\frac{\beta_N}{2} \biggl(a_{r+1}^{s}-2a_{r}^{s} +a_{r-1}^{s}\biggr), \\
\quad r = 0, \ldots, N-1,
\label{redap}
\end{multline}
or, equivalently,
\begin{equation}
\frac{dA^{s}}{dt}=-\frac{1}{2}\diag_{k}\bigl(\lambda_{s}\cos(\theta_{k})+
\mu_{s}\sin(\theta_{k})\bigr)^{2}A^{s} + \Lambda_{N}A^{s},
\label{elem}
\end{equation}
where $A^{s}(t)$ is the vector in $\mathbb{C}^{N}$ with components $a^s_1(t), \ldots, a^s_N(t)$, which are the coefficients of the solution of the restriction to $SE(2,N,K)$ of the semi-discrete equation.
The matrix $\Lambda_{N}$ is the infinitesimal generator of the jump process
with parameter $\beta_N$. The system of differential equations
\eqref{redap} is equipped with the initial condition
$a_{r}^{s}(0) = a_{r,\lambda_{s},\mu_{s}}$ with $a_{r,\lambda_{s},\mu_{s}}$ from~\eqref{pol}.
These last two formulas are equivalent to the following Theorem.

\begin{theorem}
\label{th1} For any almost periodic polynomial initial condition in
$SE(2,N,K)$, the integration of the diffusion equation reduces to
solving a finite set of independent linear ordinary differential
equations of dimension $N$. Any continuous initial condition over a
compact subset of $SE(2,N)$ can be uniformly approximated by such a polynomial.
\end{theorem}

Differential equations \eqref{redap} over $\mathbb{C}^{N}$ are not
hard to tract numerically. They have an elliptic right hand term,
and the Crank-Nicolson method (discussed in the next section) is
recommended.

\medskip

In fact, $SE(2,N,K)$ is the space of a unitary representation of
$SE(2,N)$ which is not irreducible but splits into the direct sum of
irreducible representations:
$$
SE(2,N,K) = \bigoplus_{\omega\in K}SE(2,N,\{\omega\}).
$$

This fact suggests that we can use our knowledge of the dual $\SEN$
to reduce the computations.

\begin{theorem}
In Theorem \ref{th1}, if some points of $K$ can be deduced one from
the other by elementary rotations $R_{r}$ then it is enough to
compute the resolvents corresponding to a single among these points.
\end{theorem}

\begin{proof}

Let $A^{s}$ be the solution of Equation \eqref{elem}
in $SE(2,N,\{\omega\})$ for some $\omega\in K$ and
$A'$ be the solution of the same equation in $SE(2,N,\{\omega'\})$ for $\omega' \in K$
such that $\omega' = R_{r}\omega$.
It is easy to check that $S^{r}A'$ also satisfies Equation \eqref{elem}:
\begin{equation*}
\frac{dS^{r}A'}{dt} =
-\frac{1}{2}\diag_{k}\bigl(\lambda_{s}\cos(\theta_{k})+\mu_{s}\sin(\theta_{k})\bigr)^{2}S^{r}A' +
\Lambda_{N}S^{r}A'.
\end{equation*}
So that it is not necessary to compute all the resolvents relative
to each unitary irreducible representation. It is enough to do it
for those corresponding to $\omega\in \SEN$ only.
\end{proof}

This method could be very efficient in a general setting. The last
considerations show that it is of numerical interest to put in $K$
as much as possible points in the same orbits under the elementary
rotations.

In our vision problem, the initial conditions are certainly very far from being almost periodic 
due to the lifting procedure, which lifts to a distribution. 
There are at least two ways to overcome this obstacle. 
The first way is to modify the lifting procedure, which is out of the scope of this paper. 
Here we have chosen another method based upon the spatial discretization, presented in the next section. 


\section{Numerical treatment of the diffusion equation\label{algo}}

In fact, images being given under the guise of a square table of
real values (the grey levels), we have chosen to deal with periodic
images over a basic rectangle, and to discretize with respect to the
naturally discrete $(x,y)$ variables. We take a mesh of $M_{x}\times
M_{y}$ points on the $(x,y)$-plane, and the number of angles is $N$.

\begin{remark}
In all the results we show, $M_{x}=M_{y}=M=256$,  and
the number of angles $N \approx 30$. \textbf{We don't see any
significant improvement for larger $N$}.
\end{remark}

Therefore, a function $\psi(x,y,\theta)$ over $SE(2,N)$ is
approximated by a $M^{2}\times N$ table $(\psi_{k,l}^r)$,
$r=1,\ldots,N$, $k,l=1,\ldots,M$. Hence, without loss of generality,
we set for the discretization steps $\Delta {x} = \Delta {y} =
\sqrt{M}$ and the mesh points are $x_{k}=\frac{k-1}{\sqrt{M}}$,
$y_{l}=\frac{l-1}{\sqrt{M}}$. In this section, due to the
periodicity, the upper index takes natural values modulo $N$ and the
lower indices take natural values modulo~$M$.

Remind that the semi-discrete diffusion equation over $SE(2,N)$ is~\eqref{FKP}. For numerical solution of this equation we replace the differential operators $\frac{\partial}{\partial x},\frac{\partial}{\partial y}$ by the discrete operators $D_x,D_y$, which act on $\mathbb{C}^{M^{2}}$:
\begin{align*}
D_{x}(\psi^r_{k,l})=\frac{\psi_{k+1,l}^r-\psi_{k-1,l}^r}{x_{k+1}-x_{k-1}}=
\frac{\sqrt{M}}{2}\bigl(\psi_{k+1,l}^r-\psi_{k-1,l}^r\bigr),  \\
D_{y}(\psi^r_{k,l})=\frac{\psi_{k,l+1}^r-\psi_{k,l-1}^r}{y_{l+1}-y_{l-1}}=
\frac{\sqrt{M}}{2}\bigl(\psi_{k,l+1}^r-\psi_{k,l-1}^r\bigr).
\end{align*}

\begin{remark}
Notice that this space discretisation has already been proposed in \cite{F1,FD} with numerical improvement that are out of the scope of this paper.
Moreover, the following resulting algorithm is closely related to previous work in \cite{DvA}.
\end{remark}

Then we get the diffusion equation with totally discretized
(i.e., with discretized space and angle variables) operator $D$,
which acts on $\mathbb{C}^{N}\otimes\mathbb{C}^{M^{2}}\simeq
\mathbb{C}^{N \cdot M^{2}}$:
\begin{multline}
\frac{d \psi_{k,l}^r}{dt} = \frac{1}{2} D(\psi_{k,l}^r), \quad \textrm{where}  \\
D(\psi_{k,l}^r) =
\bigl(\cos(e_r)D_{x}+\sin(e_r)D_{y}\bigr)^{2}(\psi_{k,l}^r) + \beta_N
(\psi_{k,l}^{r-1} -2\psi_{k,l}^{r} +\psi_{k,l}^{r+1}).
\label{FTE}
\end{multline}
The initial condition for \eqref{FTE} is the discrete analog of the
function $\ff(x,y,\theta)$ obtained after the lift of the original
image $f(x,y)$.

\begin{remark}
\label{number}For $M=256$, $N=30$ \eqref{FTE} is a fully coupled
linear differential equation in $\mathbb{R}^{K}$, $K=1,996,080$.
Applying an implicit or semi-implicit finite difference scheme, one
need to solve a system of $K^2 \approx 3.6\times 10^{12}$ linear
algebraic equations.
\end{remark}


\begin{figure}[h!]
\begin{center}
\includegraphics[width=320pt,height=315pt]{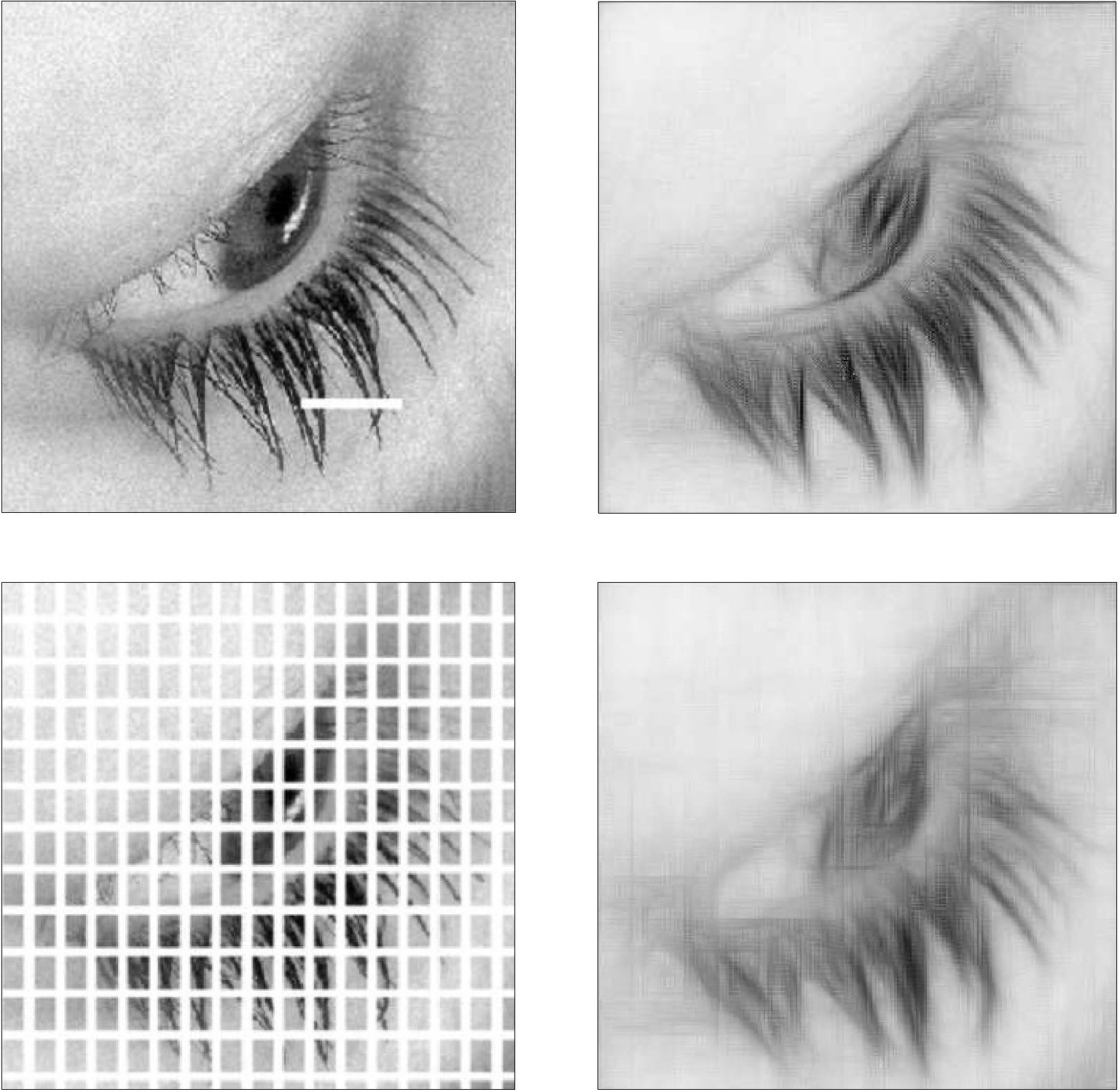}
\caption{
The initial corrupted images (left) and the images reconstructed via the hypoelliptic diffusion 
\eqref{contdif-weight} 
with $\alpha=0.25$ and time $T=0.15$ (right, up), $T=0.45$ (right, down). 
}\label{Fig1}
\end{center}
\end{figure}

\begin{theorem}
\label{th2} The solution of Equation \eqref{FTE} can be computed
\textbf{exactly} by solving in parallel $M^2$ linear differential
equations in dimension~$N$.
\end{theorem}

\begin{proof}
As previously, it is natural to apply the Fourier transform over the
abelian group ($\mathbb{Z}/M\mathbb{Z)}^{2}$, which can be computed
\textbf{exactly} by the standard FFT (Fast Fourier Transform) algorithm.
Let us denote by $\widehat{\psi}$ the Fourier transform of~$\psi$:
$$
\widehat{\psi}_{k,l}^{r}=\frac{1}{M}\sum_{p,q=1}^{M}\psi_{p,q}^{r}
\exp \biggl(-2\pi i\biggl(\frac{(k-1)(p-1)}{M}+\frac{(l-1)(q-1)}{M}\biggr)\biggr).
$$

A straightforward computation shows that
$$
\widehat{D_{x}\bigl(\psi_{k,l}^{r}\bigr)} =
i\sqrt{M}\sin\biggl(2\pi\frac{k-1}{M}\biggr)\widehat{\psi}_{k,l}^{r},
\quad \widehat{D_{y}\bigl(\psi_{k,l}^{r}\bigr)} =
i\sqrt{M}\sin\biggl(2\pi\frac{l-1}{M}\biggr)\widehat{\psi}_{k,l}^{r},
$$
and the Fourier transform maps the operator
$\cos(e_r)D_{x}+\sin(e_r)D_{y}$ to the multiplication operator
$\widehat{\psi}_{k,l}^{r} \,\to \,i\sqrt{M}\,a_{k,l}^{r} \widehat{\psi}_{k,l}^{r}$,
where
$$
a_{k,l}^{r} = \cos(e_r)\sin\biggl(2\pi\frac{k-1}{M}\biggr) +
\sin(e_r)\sin\biggl(2\pi\frac{l-1}{M}\biggr).
$$

Hence, the diffusion equation \eqref{FTE} is mapped to the following
completely decoupled system of $M^{2}$ linear differential equations
over $\mathbb{C}^{N}$:
\begin{equation}
\frac{d\widehat{\psi}_{k,l}}{dt} = A_{k,l} \,\widehat{\psi}_{k,l},  \quad
\widehat{\psi}_{k,l} = (\widehat{\psi}_{k,l}^1, \ldots, \widehat{\psi}_{k,l}^N)^*,
\label{deceq}
\end{equation}
where $A_{k,l} = \frac{1}{2} \bigl(\Lambda_{N} - M \diag_r (a_{k,l}^r)^{2}\bigr)$, $r = 1, \ldots, N$,
and $*$ means the transpose operation.
Due to periodicity in $r$, $\Lambda_{N}$ is an almost tridiagonal $N \times N$ matrix:
it contains two extra non-zero elements in the right-up and left-low corners.

Therefore the solution of \eqref{deceq} is
\begin{equation}
\widehat{\psi}_{k,l}(t) = e^{t A_{k,l}} \widehat{\psi}_{k,l}(0),
\label{soleq}
\end{equation}
where the initial function $\widehat {\psi}_{k,l}(0)$ is known. Finally,
the solution of \eqref{FTE} is the inverse Fourier transform over
($\mathbb{Z}/M\mathbb{Z)}^{2}$ of $\widehat{\psi}_{k,l}^{r}(t)$,
obtained using the inverse FFT. \end{proof}

For instance, for $M=256$ and $N=30$, this is solving $65536$ linear
differential equations in dimension~$30$.

\begin{remark}
1.\ The complexity of the FFT used twice in our algorithm (including
the inverse transform) is negligible in front of the complexity of
the numerical integration of the decoupled linear differential
equations.

2.\  The discretized diffusion~\eqref{FTE} can be
numerically integrated with the semi-explicit Crank-Nicolson scheme.
The convergence of this scheme for the considered class of equations
and some estimations are well known; see e.g.,~\cite[chapter~5]{MAR}.
Note that the matrices $A_{k,l}$ are tridiagonal plus two terms in
the right-up and left-low corners (coming from $\Lambda_{N}$ due to
periodicity). An effective algorithm for solving such linear system
is suggested in~\cite{AA}.

3.\ Formula \eqref{soleq} for solutions of the
linear differential equations \eqref{deceq} has an obvious
advantage. Once $M,N$ and $\alpha$ (or, equivalently, $\beta_N$) are
fixed, the matrices $A_{k,l}$ \textbf{are universal}, and therefore
their exponentials can be computed once for all. Moreover, using a time step
$\tau$ and the formula $e^{n\tau A_{k,l}} = \bigl(e^{\tau
A_{k,l}}\bigr)^n$, it is enough to compute only the exponential
$e^{\tau A_{k,l}}$.

4.\ In fact, it is not necessary to compute $M^{2}$ exponentials of
matrices: it is enough to compute them at the points $k,l$ that fall
in the dual space $\SEN$. This is much less. For large $N,$ is
number is of order $M$ in place of $M^{2}$. Here, the structure of
$SE(2,N)$ is crucial, and specially formula~\eqref{Discrker}.
\end{remark}


\section{Heuristic complements \label{heursec}}

It has to be noticed that:
\begin{itemize}
\item The treatment of images up to now is essentially global.
\item It is not necessary to know where the image is corrupted.
\end{itemize}
Presumably, the visual cortex V1 is also able in some cases to
detect that the image is corrupted at some place, and to take this
a-priori knowledge into account.

We tried to investigate methods that would improve on our algorithm
in this direction. These methods are based upon the general idea of
distinguishing between the so called ``good'' and ``bad'' points (pixels) of the
image under reconstruction.

Roughly speaking, good points are points that are already reconstructed enough (including non-corrupted points), and we would like to preserve them from the diffusion or at least to weaken the effect of the the diffusion at these points, while on the set of bad points the diffusion should proceed normally. Exact definitions will be given below.
In this section, we present two heuristic methods:
\begin{itemize}
\item
Static restoration method, where the sets of good  and bad points do
not change along the treatment, that is, they coincide with corrupted and non-corrupted points respectively.
\item
Dynamic restoration method, where the sets of bad and good points coincide with corrupted and non-corrupted points respectively only at the initial moment, but along the treatment some bad points may become good, so that the sets of good and bad points change.
\end{itemize}

A very similar discussion has raised in the adaptation in the conductivity matrix in \cite{DF1,FD,ACHA}.

\subsection{Restoration: downstairs or upstairs?}

One natural idea would be to iterate the following steps:
\begin{itemize}
\item
Lift the plane image to $PT\mathbb{R}^{2}$.
\item
Integrate the diffusion equation for some fixed small $\tau$.
\item
Project the solution down to the plane.
\item
Restore the non-corrupted part of the initial image.
\end{itemize}

In practice, this idea does not work at all. The main reason is that
the diffusion acts on both the corrupted and the non-corrupted
parts, and there is not good coincidence at the frontier after the
restoration.

All variations of the previous idea, at the level of the plane image
do not provide acceptable results. From what we conclude that it is
necessary to proceed the restoration of the the non-corrupted part
at the level of the lifted image, i.e., on the bundle $PT\mathbb{R}^{2}$.

\subsection{Static restoration (SR)}


\begin{figure}[h!]
\begin{center}
\includegraphics[width=330pt,height=320pt]{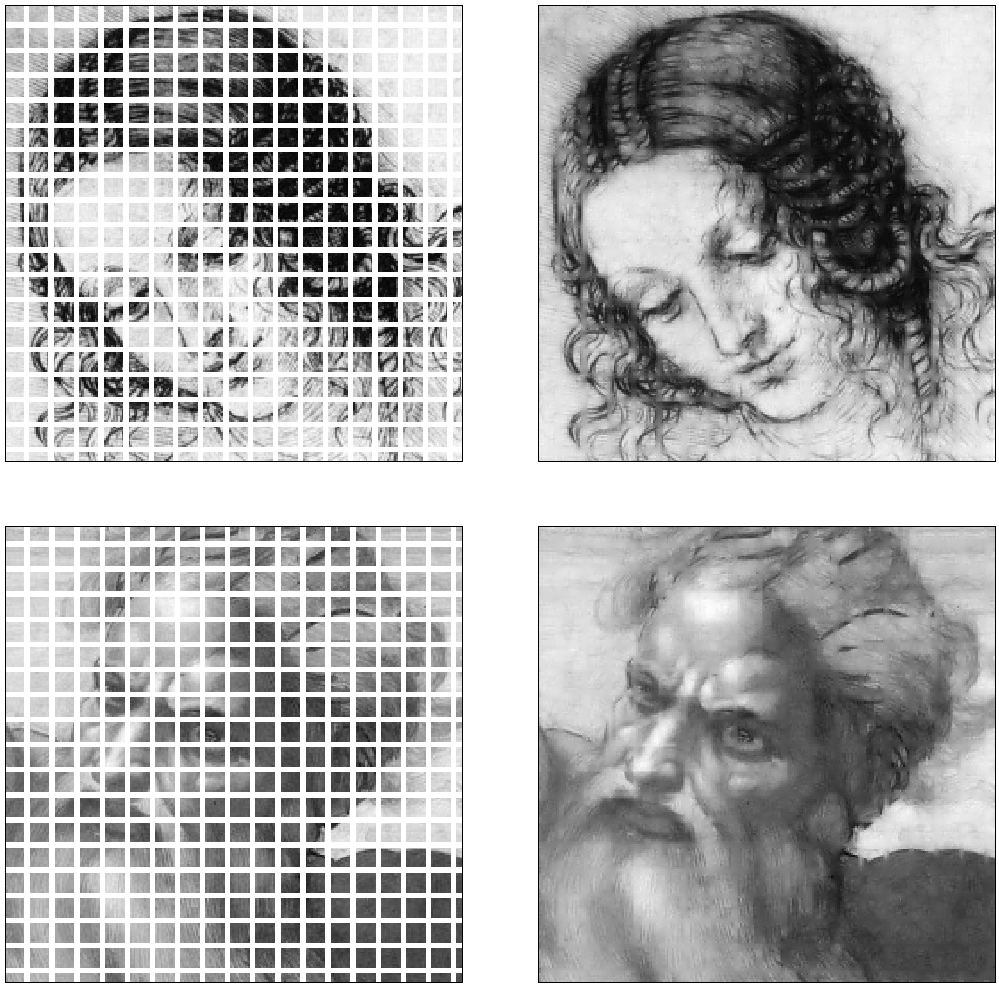}
\caption{The initial corrupted images (left) and the images reconstructed via the hypoelliptic diffusion 
\eqref{contdif-weight} 
and the SR procedure with parameters from Table~1 (right). 
}\label{Fig2}
\end{center}
\end{figure}

Assume that points $(x,y)$ of the image are separated into the set
$G$ of good (non-corrupted) and the set $B$ of bad (corrupted)
points, $f(x,y)=0$ for all $(x,y) \in B$ and $f(x,y)>0$ for all
$(x,y) \in G$. The idea of the restoration procedure is to ``mix''
the solution $\psi(x,y,\theta,t)$ of the diffusion equation with the
initial function $\psi(x,y,\theta,0)=\ff(x,y,\theta)$ at each point
$(x,y) \in G$.

The ``mixing'' is fulfilled many times during the integration of our
equation. Namely, split the segment $[0,T]$ into $n$ small intervals
with the mesh points $t_i = i\tau$, $\tau = T/n$, $i=0,1,\ldots,n$,
and proceed the integration of our equation on each
$[t_{i},t_{i+1}]$ with the initial condition
\begin{equation*}
\psi\I_{t=t_{i}} = \left \{
\begin{aligned}
\psi(x,y,\theta,t_{i}), \ \  &\textrm{if} \ \  (x,y) \in B,   \\
\sigma(x,y,t_{i})\,\psi(x,y,\theta,t_{i}), \ \  &\textrm{if} \ \  (x,y) \in G,   \\
\end{aligned}
\right.
\end{equation*}
where the function $\psi(x,y,\theta,t_{i})$ is computed after
integration on the previous interval, and the factor
$\sigma(x,y,t_{i})$ is defined by
\begin{equation*}
\sigma(x,y,t_{i}) = \frac{\epsilon h(x,y,0) + (1-\epsilon)
h(x,y,t_{i})}{h(x,y,t_{i})}, \ \quad h(x,y,t) = \max_{\theta}
\psi(x,y,\theta,t),
\end{equation*}
$0 \le \epsilon \le 1$. After that we obtain the function
$\psi(x,y,\theta,t_{i+1})$ and repeat the procedure on the next
interval.

This procedure essentially depends on two parameters that can be
chosen experimentally: the natural number $n$ (number of treatments)
and the coefficient $\epsilon$, which defines the strength of each
treatment.

\subsection{Dynamic restoration (DR)}

There is a defect in the procedure described above: different
corrupted parts of the image have different velocities of
reconstruction.
This effect can cause serious defects.


\begin{figure}[h!]
\begin{center}
\includegraphics[width=330pt,height=320pt]{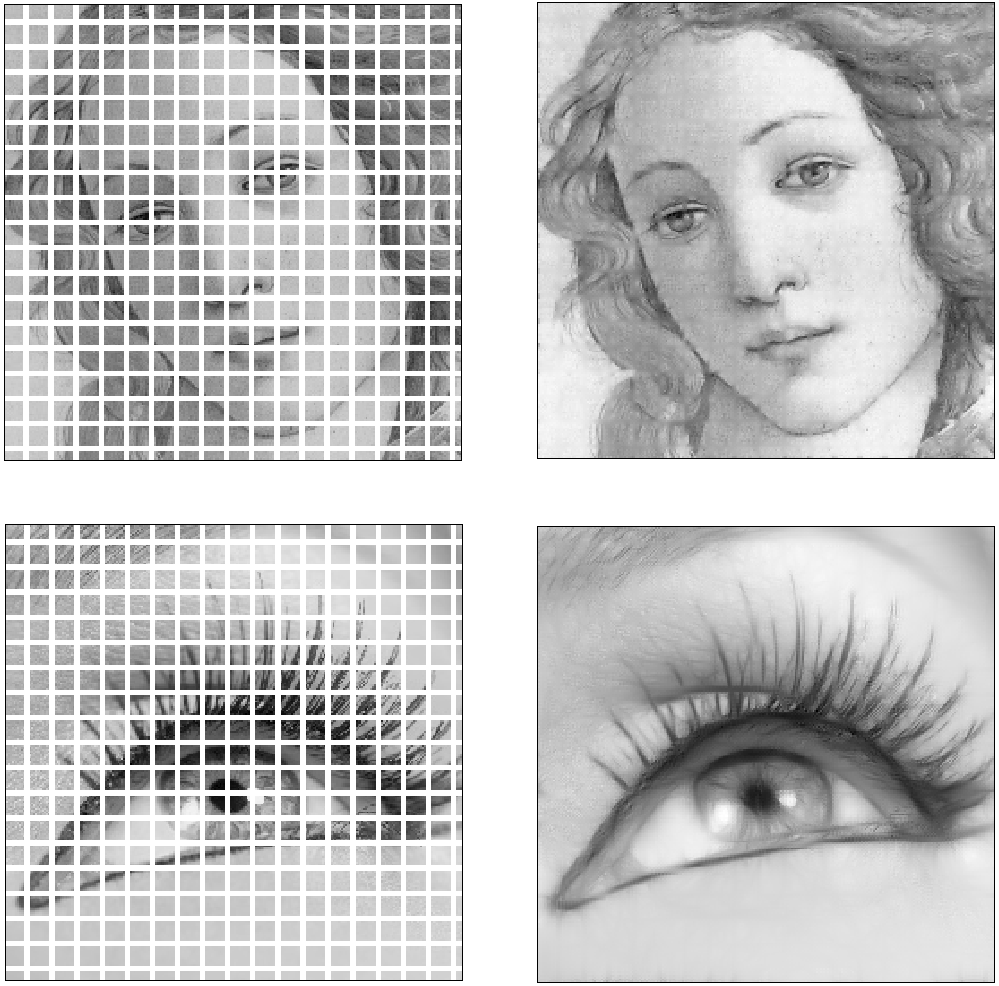}
\caption{The initial corrupted images (left) and the images reconstructed via the hypoelliptic diffusion 
\eqref{contdif-weight} 
and the DR procedure with parameters from Table~1 (right). 
}\label{Fig3}
\end{center}
\end{figure}

For instance, consider two different points $(x',y')$ and
$(x'',y'')$ of the set $B$. Let $T'$ be the time of the diffusion
required for $(x',y')$ being reconstructed well, $T''$ be the same
for $(x'',y'')$, and $T' \ll T''$. Applying the diffusion with $T
\approx T'$, we get the image not reconstructed enough at
$(x'',y'')$. On the other hand, in the case $T \approx T''$ the
image is reconstructed well at $(x'',y'')$, but a defect near the
point $(x',y')$ can appear. This brings us to a natural idea: to use
the diffusion with the time $T$ depending on a point $(x,y)$.
A simple way to realize it is to make the sets of good and bad
points changing in time so that a bad point at some moment may
become good, and the effect of the diffusion at this point is not
essential anymore.

\medskip

Define the initial set of good points $G_0$ as before (the set of
non-corrupted points of the image) and define $G_i$, $i=1, \ldots,
n$, as follows:

Denote by $f_{i-1}(x,y)$ the image at step $i-1$.
Then $G_i$ contains $G_{i-1}$ and all points $(x,y)$ of the boundary
$\partial B_{i-1}$ such that the value $f_{i-1}(x,y)$ is larger than
the average value of $f_{i-1}$
over the points of the set $B_{i-1}$ in a neighborhood of $(x,y)$.
In practice, it is convenient to choose for such a neighborhood the usual discrete 9-points neighborhood,
that is, the set of 9 points $(x_k,y_l)$, where
$x_k = x + k\Delta x$, $y_l = y + l\Delta y$, $k,l = -1,0,1$.

The set $B_i$ is defined as the complement of $G_i$.


\appendix

\section{Results with high corruption rates \label{APP6}}

Here we present a series of results of reconstruction with high corruption rates
via the hypoelliptic diffusion \eqref{contdif-weight} and the DR procedure
with different parameters found experimentally and  listed in Table~1 below.

\begin{table}[ht]
\caption{Parameters of the reconstruction for Fig. \ref{Fig2} -- \ref{FigE9}.
The images have resolution $256 \times 256$ (pixels) and
the corrupted parts consist of vertical and horizontal lines of width~$w$:
}
\begin{center}
\begin{tabular}{|c|c|c|c|c|c|c|}
\hline
\multicolumn{1}{|c|}{$\!\phantom{\Bigl|}$ Figure: $\!\phantom{\Bigr|}$} & \multicolumn{2}{|c|}{Diffusion:} & \multicolumn{2}{|c|}{Restoration:} & \multicolumn{2}{|c|}{Corruption:} \\
{}   &   $\alpha$ & $T$    &  $\phantom{o}$ $n$ $\phantom{o}$ & $\epsilon$ &   $w$, pixels  &  Total, \% \\
\hline
\ref{Fig2}, up  & 2.0 &  0.8 &  200 & 0.5  &  3  & 37 \\
\hline
\ref{Fig2}, down & 3.0 &  0.2 &  200 & 0.5  &  3  & 37 \\
\hline
\ref{Fig3}, up & 4.0 &  0.2 &  200 & 0.5  &  3  & 37 \\
\hline
\ref{Fig3}, down & 0.30 &  4.0 &  160 & 0.5 &  3  & 37\ \\
\hline
\ref{FigE1} & 0.30 &  4.0 &  160 &  0.5 &  3   & 67 \\
\hline
\ref{FigE2} & 0.30 &  4.0 &  160 &  0.5  &  4 & 58 \\
\hline
\ref{FigE3} & 0.30 &  4.0 &  160 &  0.5 &  5  & 65 \\
\hline
\ref{FigE4} & 0.40 &  4.0 &  120 &  0.5 &  5  & 69 \\
\hline
\ref{FigE5} & 0.35 &  5.0 &  160 &  0.5 &  6  & 67 \\
\hline
\ref{FigE6} & 0.33 &  6.0 &  180 &  0.5 &  7  & 68 \\
\hline
\ref{FigE7} & 0.33 &  6.0 &  250 &  0.5 &  8  & 43 \\
\hline
\ref{FigE8} & 0.33 &  6.0 &  250 &  0.5 &  8  & 53 \\
\hline
\ref{FigE9} & 0.33 &  6.0 &  250 &  0.5 &  10  & 41 \\
\hline
\end{tabular}
\end{center}
\end{table}


\begin{figure}[h!]
\begin{center}
\includegraphics[width=330pt,height=153pt]{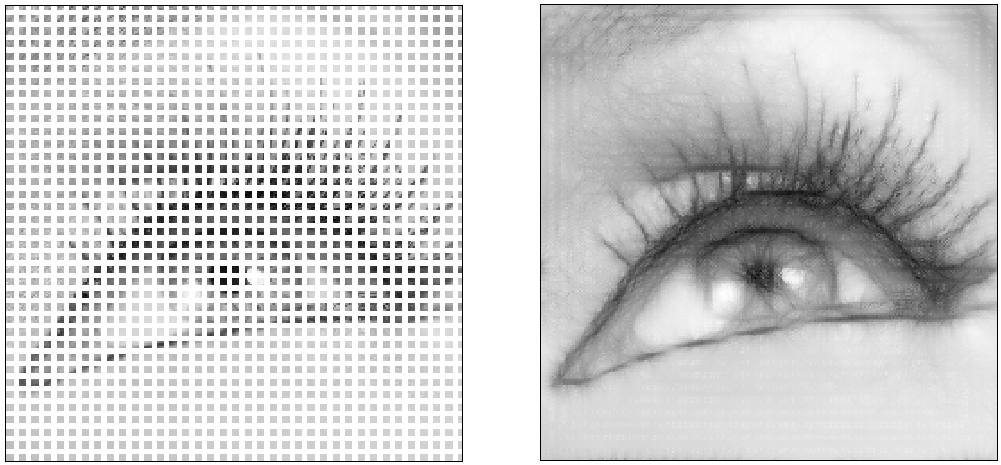}
\caption{}
\label{FigE1}
\end{center}
\end{figure}

\begin{figure}[h!]
\begin{center}
\includegraphics[width=330pt,height=153pt]{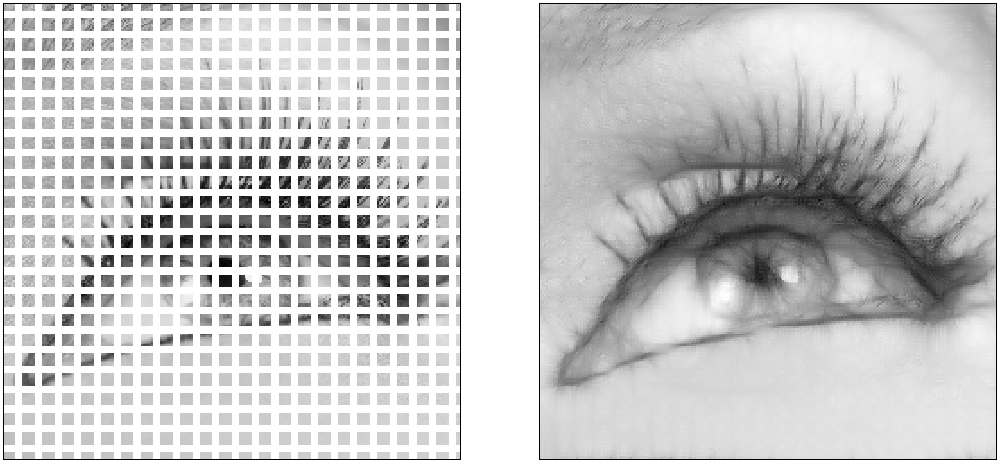}
\caption{}
\label{FigE2}
\end{center}
\end{figure}

\begin{figure}[h!]
\begin{center}
\includegraphics[width=330pt,height=153pt]{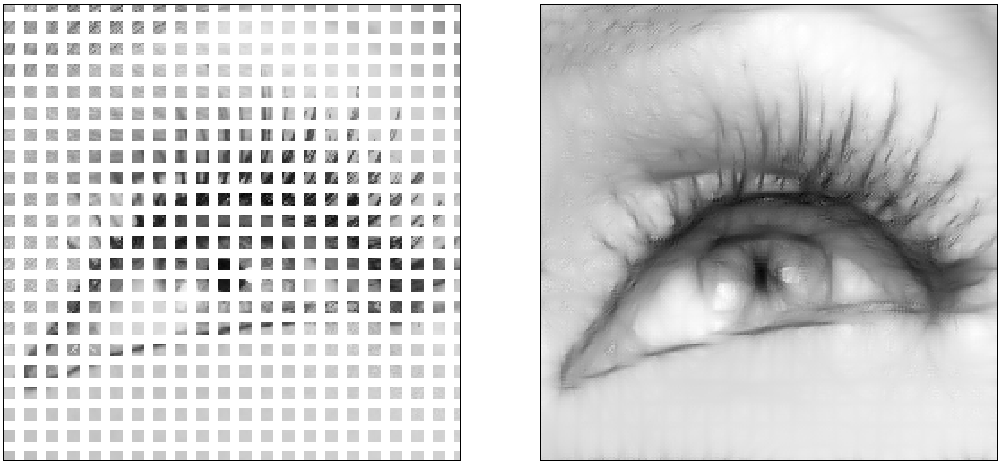}
\caption{}
\label{FigE3}
\end{center}
\end{figure}

\begin{figure}[h!]
\begin{center}
\includegraphics[width=330pt,height=153pt]{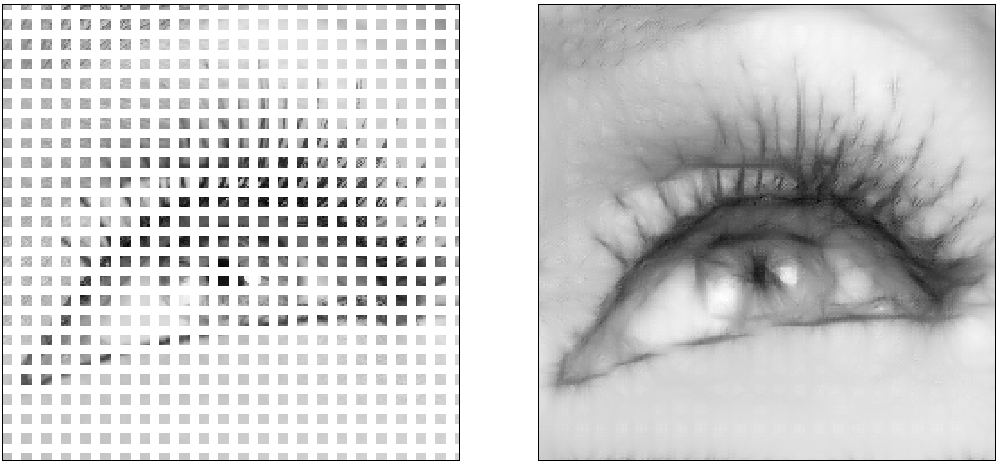}
\caption{}
\label{FigE4}
\end{center}
\end{figure}

\begin{figure}[h!]
\begin{center}
\includegraphics[width=330pt,height=153pt]{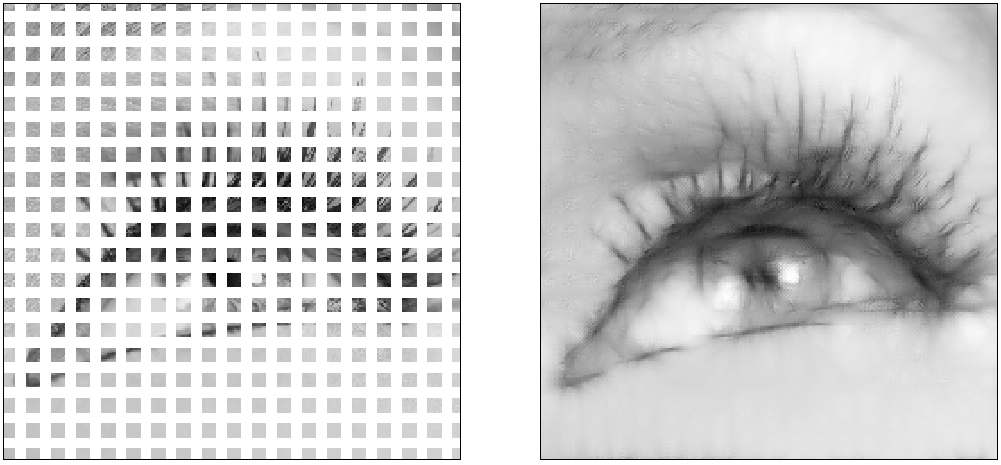}
\caption{}
\label{FigE5}
\end{center}
\end{figure}

\begin{figure}[h!]
\begin{center}
\includegraphics[width=330pt,height=153pt]{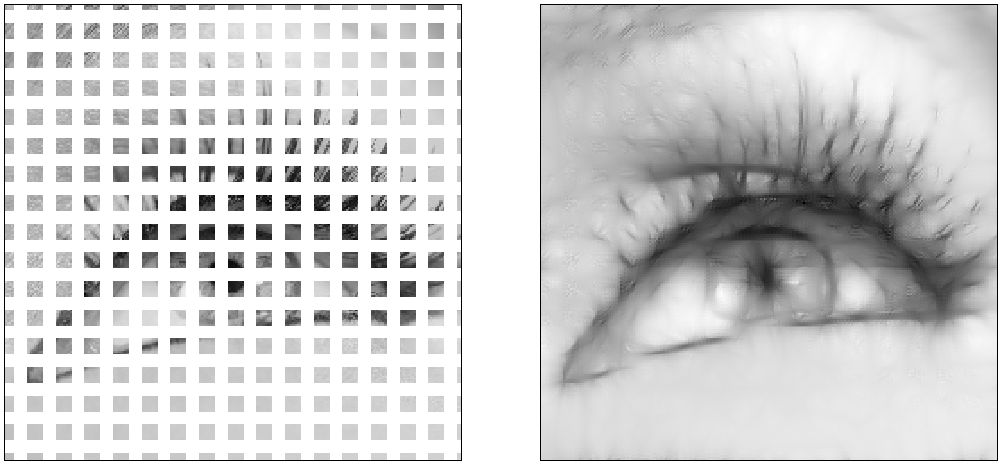}
\caption{}
\label{FigE6}
\end{center}
\end{figure}

\begin{figure}[h!]
\begin{center}
\includegraphics[width=330pt,height=153pt]{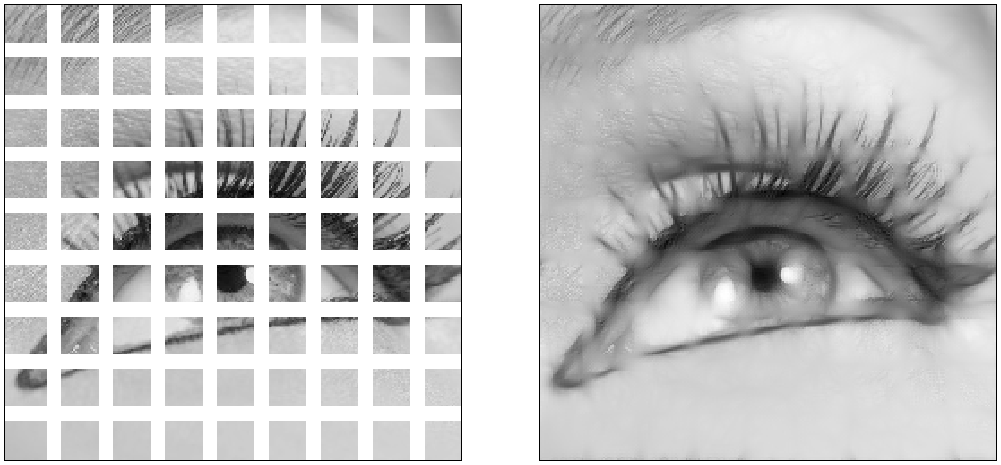}
\caption{}
\label{FigE7}
\end{center}
\end{figure}

\begin{figure}[h!]
\begin{center}
\includegraphics[width=330pt,height=153pt]{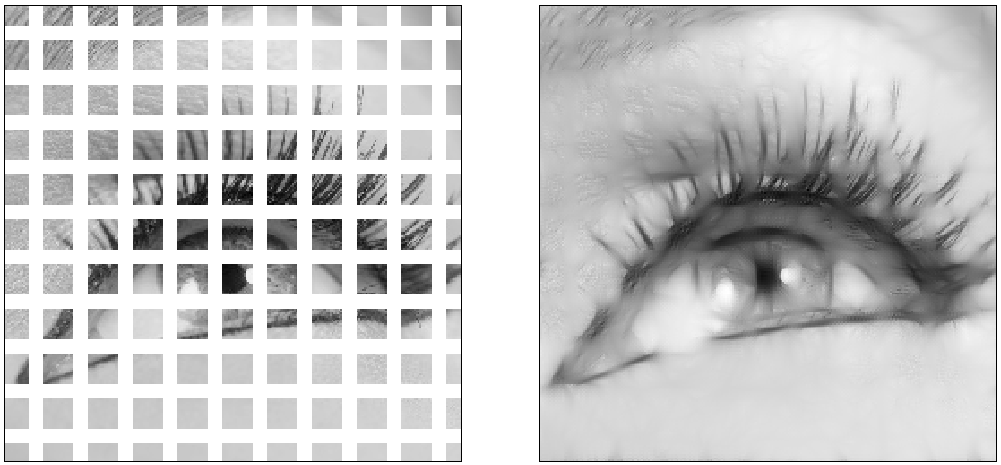}
\caption{}
\label{FigE8}
\end{center}
\end{figure}

\begin{figure}[h!]
\begin{center}
\includegraphics[width=330pt,height=153pt]{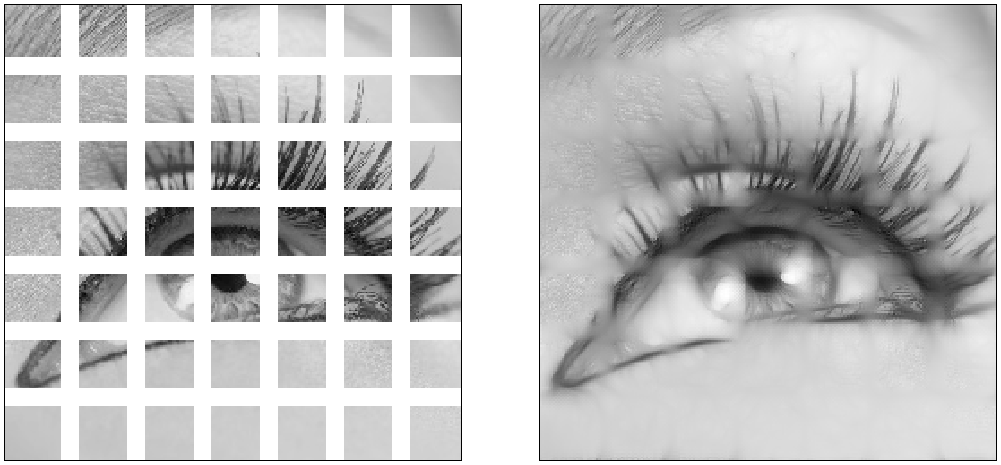}
\caption{}
\label{FigE9}
\end{center}
\end{figure}


\newpage

\section{Theoretical complements \label{APP7}}

In this appendix, we collect shortly a few results needed for the
understanding of the paper.

\subsection{The Generalized Fourier Transform (GFT)\label{GFT}}

In this appendix we recall the theory of the Generalized Fourier Transform. For an introduction including the non-unimodular case, see \cite{duflo}.

Given a locally compact unimodular topological group $G$ of
Type~I\footnote{We do not say what type~I means. We just need here
to know that both $SE(2)$ and $SE(2,N)$ are type~I. For instance,
any connected semi-simple or nilpotent group is type~I.}, the dual
$\widehat{G}$ is the set of (strongly) continuous complex unitary
irreducible representations $(\chi_{\widehat{g}},H_{\widehat{g}})$
of $G$. For $\widehat{g}\in\widehat{G},$ $\chi_{\widehat {g}} \colon
G \to U(H_{\widehat{g}})$, the unitary group of the complex Hilbert
space $H_{\widehat{g}}$. The space of complex $L^{2}$ functions over
$G$ with respect to the Haar measure is denoted by $L^{2}(g,dg)$.
The generalized Fourier transform (GFT) of $f\in L^{2}(g,dg)$ is
defined by\footnote{As usual, the integral below is well defined for
$f\in L^{1}(G,dg)$ only, but is extended by continuity.}:
\begin{equation}
\widehat{f}(\widehat{g})=\int\limits_{G}f(g)\chi_{\widehat{g}}(g^{-1})\,dg.
\label{gft}
\end{equation}

The operator $\widehat{f}(\widehat{g})$ is Hilbert-Schmidt over
$H_{\widehat{g}}$ and there is a measure $d\widehat{g}$ over
$\widehat{G}$, called the Plancherel measure, such that the GFT is
an isometry to $L^{2}(\widehat{g},d\widehat{g})$, where the space
$L^{2}(\widehat{g},d\widehat{g})$ is the continuous Hilbert sum of
the spaces of Hilbert-Schmidt operators over the spaces
$H_{\widehat{g}}$, with respect to Plancherel's measure. As a
consequence, we have the inversion formula:
\begin{equation}
f(g)=\int\limits_{\widehat{G}}\widehat{f}(\widehat{g})\chi_{\widehat{g}}(g)\,d\widehat{g}.
\label{gftinv}
\end{equation}
The GFT is a natural extension of the ordinary Fourier transform
over abelian groups and it has all the corresponding properties,
such as: mapping convolution to product, etc. (see~\cite{ABG}).

As in the case of the usual heat equation over $\mathbb{R}^{n}$, we
use it to solve our heat equations over the groups $SE(2)$ and
$SE(2,N)$.

\subsection{Bohr Compactification and Almost Periodic Functions\label{map}}

The Bohr compactification of a topological group $G$ is the
universal object $(G^{\flat},\widetilde{\sigma})$ in the category of
diagrams:
$$
\sigma \colon G \to H,
$$
where $\sigma$ is a continuous homomorphism from $G$ to a compact
group $H$. If the mapping $\widetilde{\sigma} \colon G \to
G^{\flat}$ is injective, $G$ is called maximally almost periodic (MAP).

The set $AP(G)$ of almost periodic functions over $G$ is the pull
back by $\widetilde{\sigma}$ of the set of continuous functions over
$G^{\flat}$.

The group $G$ is MAP iff the continuous unitary finite dimensional
representations of $G$ separate the points. A connected locally
compact group is MAP iff it is the direct product of a compact group
by $\mathbb{R}^{n}$. The group $SE(2,N)$ is MAP, while $SE(2)$
is not.

A continuous function $f\in AP(G)$ iff its right (or left)
translated form a relatively compact subset of $E(G)$, the set of
bounded continuous functions over $G$, iff it is a uniform limit of
coefficients of unitary irreducible representations of $G$. If $G$
is MAP, $AP(G)$ is dense in the space of continuous functions over
$G$, in the topology of uniform convergence over compact sets.

For (Chu) duality over MAP groups, see \cite{CHU} and the book~\cite{HEY}.
For introduction to almost periodic functions, see the original
paper~\cite{Weil} and the nice exposition~\cite{DIX}.

\subsection{Proof that expressions (2.10) and (2.12) are identical\label{kerproof}}

In this section, all integer indices take values between $1$ and $N$, and addition is always modulo~$N$.

Set $M_{\lambda,\nu}(t) = e^{\widetilde{A}_{\lambda,\nu}t}$, where
the matrix $\widetilde{A}_{\lambda,\nu}$  is defined
in~\eqref{restf}. Then we need to establish the identity of the two expressions:
\begin{align*}
D_{t}^1(z,e_{r})&= \! \int\limits_{\SEN} \! \trace
\Bigl(M_{\lambda,\nu}(t) \cdot \diag_{k}\Bigl(e^{i \langle
V_{\lambda,\nu},R_{k}z \rangle }\Bigr) S^{r} \Bigr) \lambda d\lambda
d\nu, \\
D_{t}^2(z,e_{r})&= \! \int\limits_{\mathbb{R}^2} \!
\bigl(M_{\lambda,\nu}(t) \delta_N \bigr)_r \, e^{i \langle
V_{\lambda,\nu},z \rangle} \lambda d\lambda d\nu.
\end{align*}

Set
$$
M^r_{\lambda,\nu}(t) = e^{\widetilde{A}^r_{\lambda,\nu}t},  \quad
\widetilde{A}^r_{\lambda,\nu} = \Lambda_{N} - \diag_{k}
\bigl(\lambda^{2}\cos^{2}(e_{k+r}-\nu)\bigr).
$$
The following fact is crucial:
\begin{equation}
S^{-r} M_{\lambda,\nu}(t) S^{r} = M^r_{\lambda,\nu}(t), \label{Xc}
\end{equation}
where $S$ is the shift matrix over $\mathbb{C}^{N}$ (i.e.,
$Se_{k}=e_{k+1}$ for $k=1, \ldots, N-1$, and $Se_{N}=e_{1}$). The
equality \eqref{Xc} follows from $S^{-r}
\widetilde{A}^r_{\lambda,\nu} S^{r} = \widetilde{A}_{\lambda,\nu}$,
a consequence of $S^{-r} \Lambda_N S^{r} = \Lambda_N$ and of the
general relation
\begin{equation}
(S^{-r} B S^{r})_{n,m}=B_{n-r,m-r}, \label{Xd}
\end{equation}
which holds true for any $N \times N$ matrix $B$.

An immediate computation shows that
\begin{multline}
D_{t}^1(z,e_{r}) = \! \int\limits_{\SEN} \! \sum_{n}
\Bigl(M_{\lambda,\nu}(t) \cdot \diag_{k}\Bigl(e^{i \langle
V_{\lambda,\nu},R_{k}z \rangle }\Bigr) \Bigr)_{n+r,n} \lambda
d\lambda d\nu = \\
\! \int\limits_{\SEN} \! \sum_{n}
\bigl(M_{\lambda,\nu}(t)\bigr)_{n+r,n} e^{i \langle
V_{\lambda,\nu},R_{n}z \rangle } \lambda d\lambda d\nu. \label{AA1}
\end{multline}

On the other hand, using the relations \eqref{Xc} and \eqref{Xd}, we
have:
\begin{multline*}
D_{t}^2(z,e_{r}) =
\! \int\limits_{\mathbb{R}^2} \!
\bigl(M_{\lambda,\nu}(t) \delta_N \bigr)_r \, e^{i \langle
V_{\lambda,\nu},z \rangle} \lambda d\lambda d\nu = \!
\int\limits_{\mathbb{R}^2} \! \bigl(M_{\lambda,\nu}(t)\bigr)_{r,N}
\, e^{i \langle V_{\lambda,\nu},z \rangle} \lambda d\lambda
d\nu = \\
\! \int\limits_{\SEN} \! \sum_{n}
\bigl(M^n_{\lambda,\nu}(t)\bigr)_{r,N} \, e^{i \langle
V_{\lambda,{\nu-e_n}},z \rangle} \lambda d\lambda
d\nu = \\
\! \int\limits_{\SEN} \! \sum_{n} \bigl(S^{-n} M_{\lambda,\nu}(t)
S^{n}\bigr)_{r,N} \, e^{i \langle V_{\lambda,{\nu}}, R_{-n}z
\rangle} \lambda d\lambda d\nu = \\
\! \int\limits_{\SEN} \! \sum_{n} \bigl(M_{\lambda,\nu}(t)
\bigr)_{r-n,N-n} \, e^{i \langle V_{\lambda,{\nu}}, R_{-n}z \rangle}
\lambda d\lambda d\nu.
\end{multline*}
Finally, changing $n$ for $-n$ and observing that $N+n=n$, we get
the following expression:
\begin{equation*}
D_{t}^2(z,e_{r}) = \! \int\limits_{\SEN} \! \sum_{n}
\bigl(M_{\lambda,\nu}(t) \bigr)_{n+r,n} \, e^{i \langle
V_{\lambda,{\nu}}, R_{n}z \rangle} \lambda d\lambda d\nu,
\end{equation*}
which is exactly \eqref{AA1}.


\begin{acknowledgement}
The authors thank Jean Petitot for his valuable comments and
suggestions, and Prof. Eric Busvelle for his help.
\end{acknowledgement}

\begin{acknowledgement}
This research has been supported by the Europen Research Council, ERC StG 2009 ``GeCoMethods'' (contract 239748),
by the ANR ``GCM'' (programme blanc NT09-504490), and by the DIGITEO project CONGEO.
\end{acknowledgement}


\end{document}